\documentclass[12pt]{amsart}

\usepackage{enumerate, amsmath, amsthm, amsfonts, amssymb, xy,  mathrsfs, graphicx, paralist, fancyvrb, eucal}
\usepackage[usenames, dvipsnames]{xcolor}
\usepackage[margin=1in]{geometry} 
\usepackage[bookmarks, colorlinks=true, linkcolor=blue, citecolor=blue, urlcolor=blue]{hyperref}

\input xy
\xyoption{all}

\numberwithin{equation}{section}
\newtheorem{theorem}[equation]{Theorem}

\newtheorem{proposition}[equation]{Proposition}
\newtheorem{lemma}[equation]{Lemma}
\newtheorem{corollary}[equation]{Corollary}

\theoremstyle{definition}
\newtheorem{rmk}[equation]{Remark}
\newenvironment{remark}[1][]{\begin{rmk}[#1] \pushQED{\qed}}{\popQED \end{rmk}}
\newtheorem{eg}[equation]{Example}

\newtheorem{defn}[equation]{Definition}
\newenvironment{definition}[1][]{\begin{defn}[#1]\pushQED{\qed}}{\popQED \end{defn}}

\newenvironment{subeqns}[1][]{\addtocounter{equation}{-1}
\begin{subequations}

}{\end{subequations}}

\makeatletter
\@addtoreset{equation}{section}
\renewcommand{\thesubsection}{%
  \ifnum\c@subsection<1 \@arabic\c@section
  \else \thesection.\@arabic\c@subsection
  \fi
}
\makeatother

\newcommand{\bC}{\mathbf{C}}
\newcommand{\cC}{\mathcal{C}}

\newcommand{\cE}{\mathcal{E}}

\newcommand{\rE}{\mathrm{E}}
\newcommand{\sE}{\mathscr{E}}
\newcommand{\bF}{\mathbf{F}}
\newcommand{\cF}{\mathcal{F}}

\newcommand{\bG}{\mathbf{G}}
\newcommand{\cG}{\mathcal{G}}

\newcommand{\rH}{\mathrm{H}}

\newcommand{\rJ}{\mathrm{J}}

\newcommand{\bK}{\mathbf{K}}

\newcommand{\cL}{\mathcal{L}}

\newcommand{\cM}{\mathcal{M}}

\newcommand{\cO}{\mathcal{O}}

\newcommand{\bP}{\mathbf{P}}

\newcommand{\bQ}{\mathbf{Q}}
\newcommand{\cQ}{\mathcal{Q}}

\newcommand{\bR}{\mathbf{R}}
\newcommand{\cR}{\mathcal{R}}

\newcommand{\rR}{\mathrm{R}}

\newcommand{\cT}{\mathcal{T}}

\newcommand{\rU}{\mathrm{U}}

\newcommand{\cV}{\mathcal{V}}

\newcommand{\bZ}{\mathbf{Z}}

\newcommand{\fe}{\mathfrak{e}}
\newcommand{\re}{\mathrm{e}}

\newcommand{\fg}{\mathfrak{g}}

\newcommand{\fh}{\mathfrak{h}}
\newcommand{\rh}{\mathrm{h}}

\newcommand{\bk}{\mathbf{k}}

\renewcommand{\phi}{\varphi}

\renewcommand{\tilde}[1]{\widetilde{#1}}
\newcommand{\ol}[1]{\overline{#1}}

\newcommand{\arxiv}[1]{\href{http://arxiv.org/abs/#1}{{\tt arXiv:#1}}}

\makeatletter
\def\Ddots{\mathinner{\mkern1mu\raise\p@
\vbox{\kern7\p@\hbox{.}}\mkern2mu
\raise4\p@\hbox{.}\mkern2mu\raise7\p@\hbox{.}\mkern1mu}}
\makeatother

\DeclareMathOperator{\im}{image}

\renewcommand{\hom}{\operatorname{Hom}}

\DeclareMathOperator{\trace}{Tr}

\DeclareMathOperator{\rank}{rank}
\DeclareMathOperator{\ext}{Ext}

\DeclareMathOperator{\Pf}{Pf}
\DeclareMathOperator{\Sym}{Sym}

\DeclareMathOperator{\Aut}{Aut}

\DeclareMathOperator{\Spec}{Spec}

\DeclareMathOperator{\Jac}{Jac}
\DeclareMathOperator{\Pic}{Pic}

\newcommand{\GL}{\mathbf{GL}}
\newcommand{\SL}{\mathbf{SL}}

\newcommand{\Gr}{\mathbf{Gr}}

\newcommand{\fsl}{\mathfrak{sl}}
\newcommand{\fgl}{\mathfrak{gl}}

\newcommand{\PGL}{\mathbf{PGL}}
\newcommand{\pgl}{\mathfrak{pgl}}
\newcommand{\SU}{\mathrm{SU}}
\newcommand{\bmu}{\boldsymbol{\mu}}

\newcommand{\Alb}{\operatorname{Alb}}
\newcommand{\su}{\mathfrak{su}}

\newcommand{\git}{/\!/}

\newcommand{\df}[1]{{\bf \textsf{#1}}}

\title{Vector bundles on genus $2$ curves and trivectors}
\date{September 5, 2017}

\author{Eric M. Rains}
\address{Department of Mathematics, California Institute of Technology}
\email{rains@caltech.edu}

\author{Steven V Sam}
\address{Department of Mathematics, University of Wisconsin, Madison}
\email{svs@math.wisc.edu}

\thanks{SS was partially supported by a Miller research fellowship and NSF DMS-1500069.}


\begin{document}

\maketitle

\begin{abstract}
Given a complex curve $C$ of genus 2, there is a well-known relationship between the moduli space of rank 3 semistable bundles on $C$ and a cubic hypersurface known as the Coble cubic. Some of the aspects of this is known to be related to the geometric invariant theory of the third exterior power of a 9-dimensional complex vector space. We extend this relationship to arbitrary fields and study some of the connections to invariant theory, which will be studied more in-depth in a followup paper.
\end{abstract}

\section{Introduction}

Let $C$ be a smooth genus $2$ curve over a field $\bk$ with a Weierstrass point $P$ defined over $\bk$. In the case when $\bk$ is algebraically closed and of characteristic $0$, the following situation has been studied in \cite{minh, ortega}. First, using $(C,P)$, we get a canonically defined divisor $\Theta$ on the Jacobian $\rJ(C)$ of degree $0$ line bundles on $C$. We correspondingly get an embedding $\rJ(C) \subset |3\Theta|^* \cong \bP^8$. Then there exists a unique cubic hypersurface in $\bP^8$, the \df{Coble cubic} of $C$, whose singular locus is $\rJ(C)$ (see \cite{barth, coble} and also \cite{beauville} for a concise proof). Furthermore, let $\SU_3(C)$ be the moduli space of semistable rank 3 bundles on $C$ with trivial determinant. Then there is a natural surjective map $\SU_3(C) \to |3\Theta|$ of degree 2 (see \S\ref{sec:genus2-gen}) whose branch locus is a sextic hypersurface. Dolgachev conjectured that these two hypersurfaces are projectively dual to one another, and this was proven in \cite{ortega} and independently in \cite{minh}.

In \cite{GS, GSW}, it was shown that the construction of the Coble cubic is closely connected to the invariant theory of the group $\SL_9(\bk)$ acting on $\bigwedge^3(\bk^9)$ and also of the group $\SL_8(\bk)$ acting on $\bigwedge^2(\bk^8)$.
In fact, the construction associates to each stable element in $\bigwedge^3(\bk^9)$ a smooth curve of genus $2$, and this construction is surjective on moduli. This naturally leads to the question of whether one can construct the trivector from the curve. This question is studied in \cite{w39-paper}, where the approach is relatively elementary, but requires one to solve an overdetermined system of linear equations to reconstruct the trivector, and it would thus be of interest to have a more direct construction.  

The construction in \cite{GS,GSW} is based on viewing an element of $\bigwedge^3(\bk^9)$ as a family of elements of $\bigwedge^2(\bk^8)$ (i.e., as a section of $\Omega^2(3)$ on $\bP^8$), but we can equally well view it as a family of elements of $\bigwedge^3(\bk^8)$ (i.e., as a section of $\bigwedge^3 \cQ$ where $\cQ$ is the tautological quotient bundle on $\bP^8$).  In characteristic $0$, the stabilizer of a generic element of $\bP(\bigwedge^3(\bk^8))$ in $\PGL_9(\bk)$ can be shown to be (geometrically) isomorphic to $\Aut(\fsl_3)$ (this was shown implicitly in \cite[\S 6.2]{luna-richardson}, where the projective normalizer in $\SL_8$ was computed; for an alternate proof valid in arbitrary characteristic, see Proposition~\ref{prop:alpha-stab} and Corollary~\ref{cor:sl-dense} below), and to be the unique trivector stabilized by that group.  Since this group has two components, this induces a natural double cover of $\bigwedge^3(\bk^8)$, which in our setting pulls back to a double cover of $\bP^8$, which furthermore comes with a (generic) identification of its tangent spaces with $\fsl_3$.  This both suggests that this double cover of $\bP^8$ should be a moduli space of rank 3 vector bundles (in particular, $\SU_3(C)$, since this is known to be a double cover of $\bP^8$), and that we should be able to reconstruct the trivector from a suitable family of trivectors in the tangent bundle of $\SU_3(C)$.  

Our aim in the present article is to make these heuristics precise and to extend these results to an arbitrary ground field $\bk$. Namely, we show that there is a degree $2$ map $\SU_3(C) \to |3\Theta|$ of degree $2$ which is branched along an irreducible sextic hypersurface. Outside of characteristic $2$, this is a reduced hypersurface and its projective dual is a cubic hypersurface whose singular locus is isomorphic to $\rJ(C)$ (Theorem~\ref{thm:coble-duality}); for characteristic $2$, see Remark~\ref{rmk:char2}. We also explain how all of these schemes may be constructed directly from a vector $\gamma \in \bigwedge^3(\bk^9)$ which is stable with respect to the action of $\SL_9(\bk)$. 

\section{Invariant theory of $\wedge^3(8)$} \label{sec:w38}

Let $V_8$ be an $8$-dimensional vector space. Consider the action of $\GL(V_8)$ on $\bigwedge^3 V_8$. There is a unique $\GL(V_8)$-invariant hypersurface if  $\bk$ has characteristic $0$, its degree is $16$, and $\GL(V_8)$ acts transitively on its complement (when $\bk = \ol{\bk}$) \cite[\S 5, Prop. 10]{satokimura}. So we can clear denominators to get a multiple that is defined over $\bZ$ (we will isolate a particular multiple in Lemma~\ref{lem:squaremod4}). We call this hypersurface the hyperdiscriminantal locus.

For any $n$, there is a natural trilinear form $\bigwedge^3 \pgl_n \to \bk$ defined by \begin{align} \label{eqn:trace-form}
(X_1,X_2,X_3) \mapsto \trace(\tilde{X}_1\tilde{X}_2\tilde{X}_3) - \trace(\tilde{X}_2 \tilde{X}_1 \tilde{X}_3)
\end{align}
where $\tilde{X}_i$ is a preimage of $X_i$ in $\fgl_n$. It is easily seen that it is skew-symmetric and that the value is independent of the choices of lifts. 

\begin{lemma} \label{lem:squaremod4}
There is a unique multiple of the degree $16$ equation with integer coefficients which is a nonzero square modulo $4$ and nonzero modulo every odd prime. Furthermore, this equation does not vanish on the trilinear form defined in \eqref{eqn:trace-form} for $n=3$.
\end{lemma}

We postpone the proof of this fact until the end of the section.

Let $f_{16}(x)$ be this multiple of the equation and write $f_{16} = f(x)^2 + 4g(x)$ where $f,g$ are nonzero polynomials with integer coefficients. Introduce a new variable $y$ of degree $8$ and consider the equation
\begin{align} \label{eqn:double}
y^2 + f(x)y - g(x) = 0.
\end{align}
This defines a double cover of $\bigwedge^3 V_8$ defined over $\bZ$ which is branched along $f_{16}(x)=0$, and which is a reduced scheme modulo any prime $p$.

Let $\fe_8$ denote the split Lie algebra of type $\rE_8$. If we coarsen the root space grading by just taking the coefficient of $\alpha_2$ (in Bourbaki notation), we get a $\bZ$-graded decomposition $\fe_8 = \bigoplus_{i=-3}^3 \fg_i$ where
\[
\fg_{-1} = \bigwedge^3 V_8^*, \qquad \fg_0 = \fgl(V_8), \qquad \fg_1 = \bigwedge^3 V_8.
\]

\begin{proposition} \label{prop:proj-dual-w37}
Let $\bk$ be a field of characteristic $0$.
\begin{enumerate}[\rm \indent (a)]
\item There are finitely many orbits of the group $\GL(V_8)$ on both $\fg_1$ and $\fg_{-1}$.
\item There is a bijection between orbit $($closures$)$ in $\fg_1$ and $\fg_{-1}$. Namely, given an orbit $O \subset \fg_1$, send it to the orbit $O' \subset \fg_{-1}$ where $O' = \{x \in \fg_{-1} \mid [x,y]=0 \text{ for all $y \in \fg_1$}\}$. If we projectivize $\fg_1$ and $\fg_{-1}$, this bijection coincides with projective duality.
\item Furthermore, if $O'$ is the orbit of highest weight vectors, then $O$ has codimension $1$.
\end{enumerate}
\end{proposition}

\begin{proof}
(a) and (b) follow from the discussion in \cite[\S 2.2.A]{tevelev}, namely Corollary 2.9 and the following text. (c) follows from \cite[Theorem 9.19]{tevelev}.
\end{proof} 

\begin{proposition} \label{prop:w38-kernel}
\begin{enumerate}[\rm \indent (a)]
\item Pick $\gamma \in \bigwedge^3 V_8$ and a $5$-dimensional subspace $U$. Pick a basis $u_1, u_2, u_3$ for $(V_8/U)^*$ and let $\gamma' = u_1 \wedge u_2 \wedge u_3$. Then $[\gamma, \gamma'] = 0$ if and only if $\gamma \in \bigwedge^3 U + \bigwedge^2 U \otimes (V_8 / U)$.

\item $\gamma$ is unstable with respect to $\SL(V_8)$ if and only if there is a $3$-dimensional space $V_3 \subset V_8^*$ such that $\gamma(v_1, v_2, -)$ is identically $0$ for any $v_1, v_2 \in V_3$. 

\item The unstable locus is an irreducible hypersurface and is set-theoretically defined by $f_{16}$.
\end{enumerate}
\end{proposition}

\begin{proof}
(a) Set $W = V_8/U$. If $P$ is the parabolic subgroup that preserves $U$, then $\bigwedge^3 V_8$ has a filtration whose quotients are $\bigwedge^3 U$, $\bigwedge^2 U \otimes W$, $U \otimes \bigwedge^2 W$, and $\bigwedge^3 W$. Since the Lie bracket $\fg_1 \otimes \fg_{-1} \to \fg_0$ is equivariant for $\GL(V_8)$, it follows immediately that if $\gamma' \in \bigwedge^3 W^*$ and $\gamma \in \bigwedge^3 U + \bigwedge^2 \otimes W$, then $[\gamma, \gamma'] = 0$. On the other hand, if we restrict the bracket to various components, we get isomorphisms $\bigwedge^3 W^* \otimes (U \otimes \bigwedge^2 W) \to U \otimes W^*$ and $\bigwedge^3 W^* \otimes \bigwedge^3 W \to \bk$. Both $U \otimes W^*$ and $\bk$ are disjoint subspaces in $\fgl(V_8) = \fg_0$, so we see that if $\gamma$ has a nonzero component in $U \otimes \bigwedge^2 W + \bigwedge^3 W$, then $[\gamma, \gamma'] \ne 0$.

(b) We can rephrase the existence of $V_3$ as saying that $\gamma \in \bigwedge^3 U + \bigwedge^2 U \otimes (V_8/U)$ where $U = (V_8^*/V_3)^*$. Pick $\gamma$ and assume that $V_3$ exists. Pick a basis $e_1, \dots, e_5$ for $U$ and extend it to a basis $\{e_1, \dots, e_8\}$ of $V_8$. Then $\gamma$ is a sum of trivectors $[i,j,k]$ where $|\{i,j,k\} \cap \{6,7,8\}| \le 1$. In particular, given the diagonal $1$-parameter subgroup $\rho(t) = (t^3,t^3,t^3,t^3,t^3, t^{-5}, t^{-5}, t^{-5})$, we have $\lim_{t \to 0} \rho(t) \cdot \gamma = 0$ (since each $[ijk]$ gets scaled either by $t^9$ or $t$), so $\gamma$ is unstable.

In particular, let $\Gr(5,V_8)$ be the Grassmannian of $5$-dimensional subspaces of $V_8$. Let $\cR_5$ be the rank $5$ tautological bundle. The trivial bundle $\bigwedge^3 V_8$ has a rank $40$ subbundle $\bigwedge^3 \cR_5 + \bigwedge^2 \cR_5 \otimes (V_8/\cR_5)$, call this bundle $\cE$. Then $\cE$ is the total space of a variety whose image under the projection $\bigwedge^3 V_8 \times \Gr(5,V_8) \to \bigwedge^3 V_8$ is contained in the unstable locus by the previous paragraph. In characteristic $0$, the image has codimension $1$ by (a) and Proposition~\ref{prop:proj-dual-w37}. In particular, the function $n \mapsto \chi(\Gr(5,V_8); \Sym^n(\cE^*))$ is a polynomial of degree $55$. By flatness, the same is true in any characteristic, so the image has codimension $1$ always.

So the image is an irreducible hypersurface and it is contained in the unstable locus. In characteristic $0$, the unstable locus is an irreducible hypersurface defined by $f_{16}$ because every invariant polynomial is of the form $\lambda f_{16}^k$ for $\lambda \in \bk$. So the two hypersurfaces coincide in this case. In general, let $R$ be a complete DVR with residue field $\bk$ and characteristic $0$ fraction field. Choose a destabilizing $1$-parameter subgroup of $\gamma$ and diagonalize it so that we can lift it to a diagonal subgroup over $R$. In particular, we can choose a lift $\tilde{\gamma}$ of $\gamma$ over $R$ which is destabilized by this subgroup. The locus of $5$-dimensional subspaces $U$ such that $\tilde{\gamma} \in \bigwedge^3 U + \bigwedge^2 U \otimes (V_8/U)$ is a closed subvariety of $\Gr(5, R^8)$. In particular, there exists such a subspace for $\gamma$ over $\bk$. This implies that the two hypersurfaces coincide in any characteristic and proves (c).
\end{proof}

\subsection{Proof of Lemma~\ref{lem:squaremod4}}

Let $\alpha$ denote the form in \eqref{eqn:trace-form}.

\begin{lemma} \label{lem:rank1-lift}
Let $X\in \pgl_n$. Then $X$ can be lifted to a rank $1$ element of $\fgl_n$ if and only if the alternating bilinear form $\alpha(X, - , - )$ has rank $2n-2$.
\end{lemma}

\begin{proof}
Pick any lift $\tilde{X} \in \fgl_n$ of $X$. Since the trace pairing on $\fgl_n$ is perfect, it follows that $[A,B]=0$ if and only if $\alpha(A,B,C)=0$ for all $C$.  In particular, $\alpha(X,-,-)$ has rank $2n-2$ if and only if the centralizer of $\tilde{X}$ in $\fgl_n$ has codimension $2n-2$, if and only if $\tilde{X}$ has an eigenspace of dimension $n-1$.  Subtracting this eigenvalue gives a rank $1$ lift of $X$.
\end{proof}

\begin{proposition} \label{prop:alpha-stab}
For $n\ge 3$, the projective stabilizer of the form $[\alpha] \in \bP(\bigwedge^3(\pgl_n)^*)$ in $\PGL(\pgl_n)$ is $\Aut(\PGL_n)$.
\end{proposition}

\begin{proof}
The stabilizer certainly contains $\Aut(\PGL_n)$, so we need simply show that there are no more automorphisms.  By Lemma~\ref{lem:rank1-lift}, it suffices to show this for the automorphism group of the rank $1$ locus in $\bP(\pgl_n)$.  Since $n\ge 3$, an element of $\pgl_n$ with a rank $1$ lift has a unique such lift, and thus the rank $1$ locus in $\bP(\fgl_n)$ embeds in $\bP(\pgl_n)$.  Moreover, since the rank $1$ locus is isomorphic to $\bP(V_n)\times \bP(V_n^*)$, the pullback of $\cO_{\bP(\fgl_n)}(1)$ is the unique $n$th root of the anti-canonical bundle.  It follows that any linear automorphism of the locus in $\bP(\pgl_n)$ lifts to a linear automorphism in $\bP(\fgl_n)$ preserving the identity. In other words, the stabilizer of the rank $1$ locus in $\bP(\pgl_n)$ is contained in the subgroup of the automorphism group of $\bP(V_n)\times \bP(V_n^*)$ preserving the natural incidence relation. As this is precisely $\Aut(\PGL_n)$, the claim follows.
\end{proof}

Note that $\pgl_n^*$ is canonically identified with $\fsl_n$: In characteristic $0$, we may use the Killing form to identify $\pgl_n$ with its dual (up to an overall scalar).  This fails in characteristic dividing $n$, however, in which case there is an $\Aut(\pgl_n)$-invariant linear functional on $\pgl_n$, but no $\Aut(\pgl_n)$-invariant element. Indeed, $\fgl_n$ has both an invariant functional $\trace$ and an invariant element $1$, and since $\trace(1)=0$ in characteristic dividing $n$, the trace functional survives.  We see that in fact the dual of $\pgl_n=\fgl_n/1$ is canonically identified with the annihilator of $1$ under the trace pairing, i.e., the Lie algebra $\fsl_n$ of traceless matrices. 

\begin{corollary} \label{cor:sl-dense}
The orbit $\GL(\fsl_3)\alpha$ is dense in $\bigwedge^3(\fsl_3)$.
\end{corollary}

\begin{proof}
By Proposition~\ref{prop:alpha-stab}, the dimension of the orbit is $\dim(\GL(\fsl_3))-\dim(\Aut(\PGL_3)) = 64 - 8 = 56 = \dim(\bigwedge^3(\fsl_3))$.
\end{proof}

For any commutative ring $R$, any element of $R[\bigwedge^3(R^8)]^{\SL_8}$ can be evaluated on any triple $(V_8,\omega,\alpha')$ where $V_8\cong R^8$, $\omega \colon R\cong \bigwedge^8(V_8)$, and $\alpha'\in \bigwedge^3(V_8)$. Choose an isomorphism $\omega_{\fsl_3} \colon \bZ \cong \bigwedge^8(\fsl_3(\bZ))$, so that $(\fsl_3,\omega_{\fsl_3},\alpha_{\fsl_3})$ is a triple of the above form. Note that $A\mapsto -A^T$ has determinant $-1$ on $\fsl_3$, and thus either choice of $\omega_{\fsl_3}$ gives an isomorphic triple.

Now, let $\iota$ be a degree $16$ element of $\bQ[\bigwedge^3(\bQ^8)]^{\SL_8}$.  By clearing denominators and eliminating content, we may ensure that $\iota$ is a primitive element of $\bZ[\bigwedge^3(\bZ^8)]^{\SL_8}$, and this determines $\iota$ up to a sign. Since $(\fsl_3,\omega_{\fsl_3},\alpha_{\fsl_3})$ projectively generates a dense orbit in all characteristics (Corollary~\ref{cor:sl-dense}), $\iota(\fsl_3,\omega_{\fsl_3},\alpha_{\fsl_3})$ is nonzero modulo all primes, and is thus a unit.  We may thus choose the representative of $\iota$ such that $\iota(\fsl_3,\omega_{\fsl_3},\alpha_{\fsl_3})=1$.

\begin{lemma}
A triple $(V_8,\omega,\alpha)$ over an algebraically closed field is equivalent to $(\fsl_3,\omega_{\fsl_3},\alpha_{\fsl_3})$ if and only if $\iota(V_8,\omega,\alpha)=1$.
\end{lemma}

\begin{proof}
Since triples equivalent to $(\fsl_3,\omega_{\fsl_3},\alpha_{\fsl_3})$ are dense in $\bigwedge^3(\bZ^8) \git \bG_m$, every triple with nonzero invariant is equivalent to one of the form $(\fsl_3,\omega_{\fsl_3},c\alpha_{\fsl_3})$ for some $c\ne 0$. It suffices to show that two such triples are equivalent if and only if they have the same invariant. Since the invariant is $c^{16}$, there is nothing to show in characteristic $2$, while in odd characteristic we need simply show that the triples with parameter $c$ and $\zeta_{16}c$ are equivalent.  Such an equivalence is given by the determinant $1$ endomorphism $A\mapsto -\zeta_{16}A^T$ of $\fsl_3$.
\end{proof}

\begin{proposition}
$\iota$ is a nonzero square modulo $4$.
\end{proposition}

\begin{proof}
Let $Y=\Spec(\bZ[\bigwedge^3(\bZ^8)][1/\iota])$. For any geometric point of $Y$, the rank $4$ locus $Z$ of the corresponding trilinear form is isomorphic to $\bP^2\times \bP^2$.  There are two isomorphism classes of line bundles on $\bP^2\times \bP^2$ with Hilbert polynomial $(k+1)(k+2)^2(k+3)/4$ relative to the ample bundle $\cO(1,1)$, and thus the corresponding subscheme of $\Pic(Z/Y)$ is an \'etale double cover $X$ of $Y$.  

Since $Y$ is affine with $\Pic(Y)=1$, $X$ is affine and has an equation of the form $y^2+fy+g=0$, where since $X$ is \'etale, $f^2-4g$ must be a unit.  The unit group is generated by $\pm 1$ and $\iota$, so $f^2-4g=\pm \iota^\ell$ for some integer $\ell$.  By clearing denominators, we may ensure that $f$ and $g$ are polynomials, making $\ell\ge 0$.  If $\ell>1$, then $f$ is congruent mod $2$ to a multiple of $\iota$, and thus we may add a suitable polynomial to $y$ to make $f$ a multiple of $\iota$, and thus $g$ a multiple of $\iota^2$, and we may eliminate the common factor.

We thus have an expression for $X$ as $y^2+fy+g=0$ where $f^2-4g$ is one of $\{1,-1, \iota,-\iota\}$. Now, consider points of $Y(\bR)$ equivalent to $\alpha_{\fsl_3}$ and $\alpha_{\su_3}$ (with an arbitrary choice of volume form in the latter case).  The first point has $\iota=1$ and has a pair of real preimages in $X$, and thus $f^2-4g$ cannot be $-1$ or $-\iota$.  In contrast, the rank $4$ locus of $\alpha_{\su_3}$ is the restriction of scalars of $\bP^2_\bC$, and thus that point has no real preimage, ruling out $f^2-4g=1$.

In particular, $\iota=f^2-4g$, which finishes the proof.
\end{proof}

\begin{remark}
By taking the integral closure of $\bZ[\bigwedge^3(V_8)]$ in $\bZ[X]$, we obtain a natural extension of this double cover to all of $\bigwedge^3(V_8)$, which again over $\bZ[1/2]$ is given by $y^2=\iota$, and over any field has ramification locus $\iota=0$.

In particular, $\iota$ defines an irreducible hypersurface in all characteristics and is reduced in characteristics different from $2$.
\end{remark}

\section{Generalities on vector bundles on a genus $2$ curve} \label{sec:genus2-gen}

Some general background for this section can be found in \cite{popa}.

Let $C$ be a smooth curve of genus $2$ over a field $\bk$ and assume that there is a $\bk$-rational Weierstrass point $P$. Let $\rJ^1(C)$ denote the space of line bundles of degree $1$ over $C$, which has a natural divisor $\Theta$ given by $\{ \cL \mid \rH^0(C; \cL) \ne 0\}$. Let $\rJ(C) = \rJ^0(C)$ be the Jacobian of $C$, i.e., the space of line bundles of degree $0$ over $C$.

Let $\SU_3(C)$ be the coarse moduli space of rank $3$ semistable vector bundles on $C$ modulo S-equivalence \cite[\S 2.5]{popa}. Let $\SU^{\rm s}_3(C)$ be the subscheme of stable vector bundles in $\SU_3(C)$. Let $\Omega^1_{\SU_3(C)}$ denote the cotangent sheaf of $\SU_3(C)$. On $\SU^{\rm s}_3(C)$, the fiber of $\Omega^1_{\SU_3(C)}$ over $\cE$ is $\rH^0(C; \pgl(\cE) \otimes \omega_C)$ \cite[Theorem 4.5.4]{HL}.

Given a vector bundle $\cE$ over $C$, define 
\[
\Theta_\cE = \{\cL \in \rJ^1(C) \mid \rH^0(C; \cE \otimes \cL) \ne 0\}.
\]

\begin{lemma} \label{lem:theta-map}
If $\cE$ is a semistable rank $3$ vector bundle with trivial determinant, then $\Theta_\cE$ is a divisor in $\rJ^1(C)$ which is linearly equivalent to $3\Theta$.
\end{lemma}

\begin{proof}
By Riemann--Roch, $\chi(\cE \otimes \cL) = 0$, so $\Theta_\cE$ is the zero locus of a determinant (see, for example, \cite[\S 3]{popa}) and hence has codimension at most $1$. By \cite[Corollaire 1.7.4]{raynaud}, there exists $\cL \in \rJ^1(C)$ with $\rH^0(C; \cE \otimes \cL) = 0$, so the codimension of $\Theta_\cE$ is exactly $1$. 

So, the locus 
\[
\{(E, \cL) \in \SU_3(C) \times \rJ^1(C) \mid \rH^0(C; \cE \otimes \cL) \ne 0\}
\]
is a divisor in $\SU_3(C) \times \rJ^1(C)$ and hence is the zero section of some line bundle whose fibers over a fixed point $[\cE]$ of $\SU_3(C)$ are the line bundles cutting out $\Theta_\cE$, so we see that they are algebraically equivalent to one another, and hence we get a map $f \colon \SU_3(C) \to \Pic(\rJ^1(C)) \cong \Jac(C)$. In the special case $\cE = \cL_1 \oplus \cL_2 \oplus \cL_3$ for line bundles $\cL_i$ (which must be of degree $0$ by semistability), we see that $\Theta_\cE$ is linearly equivalent to $\Theta_{\cL_1} + \Theta_{\cL_2} + \Theta_{\cL_3}$, which in turn is linearly equivalent to $3\Theta$ via the theorem of the square. 

We see that $f$ contracts positive dimensional subvarieties (in particular, curves). Now let $\cM$ be an ample line bundle on $\Jac(C)$. Then $f^*(\cM)$ is constant along such curves and hence cannot be ample. Finally, the Picard group of $\SU_3(C)$ is isomorphic to $\bZ$ with ample generator \cite[Corollary 3.4]{hoffmann}, so we conclude that $f^*(\cM)$ is trivial, which implies that $f$ is a constant map. Hence, $\Theta_\cE$ is linearly equivalent to $3\Theta$ for all $\cE$.
\end{proof}

Lemma~\ref{lem:theta-map} implies that there is a well-defined morphism
\begin{align*}
\theta \colon \SU_3(C) &\to |3\Theta| \cong \bP^8\\
\cE &\mapsto \Theta_\cE.
\end{align*}
Define $\cO_{\SU_3(C)}(1) = \theta^* \cO_{\bP^8}(1)$. Let $\iota$ be the hyperelliptic involution of $C$.

\begin{proposition} \label{prop:involution}
$\theta$ commutes with the involution $\cE \mapsto \iota^* \cE^*$ on $\SU_3(C)$, and this involution is not the identity map on $\SU_3(C)$.
\end{proposition}

\begin{proof}
We may assume that $\bk$ is algebraically closed.

The first statement is \cite[Proposition 4.1]{ortega}. For the second statement, we will adapt the argument in \cite[\S 4]{ortega}, which uses results only stated in characteristic $0$.

Let $\cL$ be a line bundle on $C$ of degree $1$ with $\iota^* \cL \cong \cL$, and let $\SU_2(C; \cL)$ be the moduli space of isomorphism classes of stable rank $2$ vector bundles with determinant $\cL$. As in \cite[\S 4]{ortega}, it is sufficient to find $\cF \in \SU_2(C; \cL)$ such that $\iota^* \cF \not\cong \cF \otimes \alpha$ for every $\alpha \in \rJ(C)[2]$.

First, $\SU_2(C; \cL)$ can be embedded into $\bP^5$ as the complete intersection of two quadrics (see \cite[Theorem 1]{desale-ramanan} if the characteristic is $\ne 2$, and \cite{bhosle} for characteristic $2$) and its Picard group is isomorphic to $\bZ$ \cite[Corollary 3.4]{hoffmann}. Furthermore, only one of the generators of its Picard group has global sections, and so any automorphism of $\SU_2(C; \cL)$ is trivial on its Picard group. Since $\SU_2(C; \cL)$ is linearly normal in $\bP^5$, any automorphism of $\SU_2(C; \cL)$ extends to an automorphism of $\bP^5$. We finish by using the proof of \cite[Lemma 7]{newstead}.
\end{proof}

\begin{proposition} \label{prop:picard-group}
$\cO_{\SU_3(C)}(1)$ generates the Picard group of $\SU_3(C)$, which is isomorphic to $\bZ$. Furthermore, $\SU_3(C)$ is a normal, Gorenstein, locally factorial, projective, geometrically irreducible variety and its canonical bundle is $\cO_{\SU_3(C)}(-6)$.
\end{proposition}

\begin{proof}
The fact that $\SU_3(C)$ is a locally factorial projective variety is \cite[Corollary 3.8]{hoffmann} and $\Pic(\SU_3(C)) \cong \bZ$ is \cite[Corollary 3.4]{hoffmann}. The normal and Gorenstein properties are shown in \cite[\S 2]{balaji-mehta}. The calculation of the canonical bundle was done in characteristic $0$ in \cite[\S 7.5]{drezet}, but subject to the other properties above being known, the same proof goes through.
\end{proof}

\begin{proposition} \label{prop:surj-deg2}
$\theta$ is a surjective, finite, flat morphism of degree $2$. It is branched along a degree $6$ hypersurface in $|3\Theta| \cong \bP^8$.
\end{proposition}

\begin{proof}
We may assume that $\bk$ is algebraically closed. Let $R$ be a complete DVR whose residue field is $\bk$ and whose fraction field $\bK$ is of characteristic $0$. Then there exists a smooth curve $\cC$ over $R$ such that $\cC_\bk = C$ and $\cC_\bK$ is smooth. 

From \cite[Theorem 4.1]{langer}, there is a projective scheme $M_{\cC/R}(r,d)$ of finite type over $R$ so that for any $R$-scheme $T$, $\hom(T, M_{\cC/R}(r,d))$ is the set of S-equivalence classes of families of rank $r$ semistable bundles of degree $d$ on the geometric fibers of $T \times_R \cC \to T$ which are flat over $T$. Consequently, given $\cE_\bK \in \SU_3(\cC_\bK)$, we can extend it uniquely to a flat bundle $\cE_R$ over $\cC_R$ and then reduce it to a bundle $\cE_\bk \in \SU_3(C)$. This gives a commutative square
\[
\xymatrix{
\SU_3(\cC_\bK) \ar[r] \ar[d]_-{\theta_\bK} & \SU_3(C) \ar[d]^-\theta \\
|3\Theta|_\bK \ar[r] & |3\Theta|
}.
\]

We know that $\theta_\bK$ is a surjective degree $2$ morphism \cite[\S V, Lemma 5]{laszlo}. The bottom map is surjective onto $\bk$-points, and if we pick a preimage in $|3\Theta|_\bK$ of each point in $|3\Theta|$, then the restriction of the composition $\SU_3(\cC_\bK) \to |3\Theta|_\bK \to |3\Theta|$ to points that lie over these preimages is surjective and generically of degree $2$. We deduce that $\theta$ is surjective and of degree at most $2$. To get that the degree is $2$, we use Proposition~\ref{prop:involution}.

Since $\SU_3(C)$ is Cohen--Macaulay (Proposition~\ref{prop:picard-group} gives Gorenstein), $\theta$ is a finite flat morphism and so $\theta_* \cO_{\SU_3(C)}$ is a rank $2$ vector bundle on $\bP^8$. Furthermore, $\omega_{\SU_3(C)} \cong \theta^* \omega_{\bP^8} \otimes \cR$ where the ramification divisor is a section of $\cR$. By Proposition~\ref{prop:picard-group}, we conclude that $\cR = \cO_{\SU_3(C)}(3)$, and so $\theta_* \cO_{\SU_3(C)} \cong \cO_{\bP^8} \oplus \cO_{\bP^8}(-3)$. We conclude that $\theta$ is branched along a degree $6$ hypersurface.
\end{proof}

\begin{corollary} \label{cor:verlinde}
The Verlinde formula is valid for $\SU_3(C)$ in arbitrary characteristic. More precisely, $\cO_{\SU_3(C)}(d)$ has cohomology concentrated in a single degree $i$, with $i=0$ if $d \ge 0$ and $i=8$ if $d \le -6$, and the cohomology vanishes otherwise. Its Hilbert polynomial is
\[
\chi(\cO_{\SU_3(C)}(d)) = \frac{2}{8!} (d+1)(d+2)(d+3)^2(d+4)(d+5)(d^2+6d+56).
\]
\end{corollary}

\begin{proof}
We can embed $\SU_3(C)$ as a degree $6$ hypersurface in a weighted projective space $\bP$ with $9$ variables of degree $1$ and $1$ variable of degree $3$. So it has a locally free resolution $0 \to \cO_{\bP}(-6) \to \cO_{\bP} \to \cO_{\SU_3(C)} \to 0$ and the cohomology of $\cO_{\bP}(d)$ is at most in a single degree (either $0$ or $9$). Hence, the cohomology of $\cO_{\SU_3(C)}(d)$ is at most in a single degree (either in $0$ or $8$). So its Hilbert polynomial is independent of characteristic, and can be read off from the locally free resolution.
\end{proof}

\section{Discriminants} \label{ss:construct-trivector}

As in the previous section, let $C$ be a smooth genus $2$ curve over $\bk$, and let $P \in C(\bk)$ be a Weierstrass point. Let $\rJ^1(C)$ be the space of degree $1$ line bundles over $C$ with its divisor $\Theta$ from above. From Proposition~\ref{prop:surj-deg2}, we have a degree $2$ map
\[
\theta \colon \SU_3(C) \to |3\Theta| \cong \bP^8
\]
which is branched along a degree $6$ hypersurface in $\bP^8$. Let $(\bP^8)^{\rm s}$ be the image of the stable locus in $\SU_3(C)$ under $\theta$.

Using \eqref{eqn:trace-form}, we have maps
\[
\gamma_\cE \colon \bigwedge^3 \rH^0(C; \pgl(\cE) \otimes \omega_C) \to \rH^0(C; \omega_C^{\otimes 3}) \to \rH^0(P; \omega_C^{\otimes 3}) \cong \bk
\]
where the last isomorphism is well-defined up to a global choice (independent of $\cE)$ of scalar. This gives a map
\[
\theta^* \Omega^3_{(\bP^8)^{\rm s}} \to \Omega^3_{\SU^{\rm s}_3(C)} \xrightarrow{\gamma} \cO_{\SU^{\rm s}_3(C)} = \theta^* \cO_{(\bP^8)^{\rm s}}.
\]

\begin{lemma}
The composition descends as an anti-symmetrized form on $(\bP^8)^{\rm s}$, so it gives a section of $(\bigwedge^3 \cT_{(\bP^8)^{\rm s}})(-3)$ where $\cT$ denotes the tangent sheaf. 
\end{lemma}

\begin{proof}
The involution $\cE \mapsto \iota^* \cE^*$ induces an involution on $\theta_* \cO_{\SU_3(C)}$ which we denote by $\iota$. We have an exact sequence
\[
0 \to \cO_{\bP^8} \to \theta_* \cO_{\SU_3(C)} \xrightarrow{1-\iota} \theta_* \cO_{\SU_3(C)}
\]
which shows that the quotient $\theta_* \cO_{\SU_3(C)} / \cO_{\bP^8} \cong \cO_{\bP^8}(-3)$ is the image of the map $1-\iota$.

First suppose that the characteristic of $\bk$ is different from $2$. Then $\iota$ acts as $-1$ on $\gamma$ since $\cE \mapsto \iota^* \cE^*$ acts contravariantly on maps and in particular, $\gamma$ is in the image of $1-\iota$. If the characteristic of $\bk$ is $2$, then we can write $\gamma_\cE$ as the difference of the two forms $(X_1, X_2, X_3) \mapsto \trace(\tilde{X}_1 \tilde{X}_2 \tilde{X}_3)$ and $(X_1, X_2, X_3) \mapsto \trace(\tilde{X}_2 \tilde{X}_1 \tilde{X}_3) = \trace(\tilde{X}_3 \tilde{X}_2 \tilde{X}_1)$ where we have chosen lifts $\tilde{X}_i \in \rH^0(C; \fgl(\cE) \otimes \omega_C)$ with trace $0$ to make it well-defined. These two forms are swapped under the involution on $\SU_3(C)$, so again, $\gamma$ is in the image of $1-\iota$.

So $\gamma$ descends to a form $\Omega^3_{(\bP^8)^{\rm s}} \to \cO_{(\bP^8)^{\rm s}}(-3)$, i.e., a section of $(\bigwedge^3 \cT_{(\bP^8)^{\rm s}})(-3)$.
\end{proof}

The complement of the stable locus in $\SU_3(C)$ has codimension $3$, and so its image in $\bP^8$ has codimension $\ge 3$. In particular, given a vector bundle on $\bP^8$, any section of it  over $(\bP^8)^{\rm s}$ extends to a section on all of $\bP^8$. So we get a section of $(\bigwedge^3 \cT_{\bP^8})(-3)$. Finally, $\rH^0(\bP^8; (\bigwedge^3 \cT_{\bP^8})(-3)) = \bigwedge^3(\rH^0(\bP^8; \cO(1))^*)$, so we get 
\[
\gamma_{(C,P)} \in \bigwedge^3(\rH^0(\SU_3(C); \Theta)^*).
\]

\begin{definition} \label{def:Dgamma}
The fiber of $(\bigwedge^3 \cT_{\bP^8})(-3)$ over a point of $\bP^8$ is of the form $\bigwedge^3(\bk^8)$; in \S\ref{sec:w38}, we identified a $\GL_8$-invariant hypersurface in $\bigwedge^3(\bk^8)$, which we called the hyperdiscriminant. Given a section $\gamma \in \rH^0(\bP^8; (\bigwedge^3 \cT_{\bP^8})(-3))$, let $D_\gamma$ be the subscheme of points where the corresponding element in the fiber has vanishing hyperdiscriminant. 
\end{definition}

Locally, the hyperdiscriminant is a degree $16$ equation, and in basis-independent terms, the line that it spans is given by $(\det \bk^8)^6 \subset \Sym^{16}(\bigwedge^3(\bk^8))$. So the equation defining $D_\gamma$ is a section of $(\det \cT_{\bP^8})^6(-48) = \cO_{\bP^8}(6)$. As an aside, we can interpret $\bP^8$ as parametrizing lines in a $9$-dimensional vector space, so it has a tautological rank $8$ quotient bundle $\cQ$, and $\cT_{\bP^8} \cong \cQ \otimes \cO_{\bP^8}(1)$.

The following two statements are the main results about $\gamma_{(C,P)}$ for this section. Their proofs are given below in separate subsections.

\begin{proposition} \label{prop:inverse1}
$D_{\gamma_{(C,P)}}$ is the branch locus of $\theta$.
\end{proposition}

\begin{proposition} \label{prop:gamma-stable}
Pick a basis for $\rH^0(\SU_3(C); \Theta)^*$. Then $\gamma_{(C,P)} \in \bigwedge^3 \bk^9$ is a stable element with respect to the natural action of $\SL_9(\bk)$.
\end{proposition}

\subsection{Proof of Proposition~\ref{prop:inverse1}}

\begin{lemma} \label{lem:lift-char2}
Let $C$ be a hyperelliptic curve over an algebraically closed field with hyperelliptic involution $\iota$, and let $L$ be a line bundle such that $\iota^*L\cong L$ and $\chi(L) = 0$. Then there is a lift $(\tilde{C},\tilde{L})$ of this pair to characteristic $0$ such that $\rh^0(\tilde{L}) = \rh^0(L)$.
\end{lemma}

\begin{proof}
If $\rh^0(L)=0$, any lift of $L$ satisfies the condition.  Now, suppose $\rh^0(L)=1$.  Then $L$ is represented by a unique effective divisor, which must therefore be $\iota$-invariant. If the support of this divisor contains a non-Weierstrass point, or a Weierstrass point with multiplicity greater than $1$, then we can replace that subdivisor by any representative of the hyperelliptic class $\eta$, contradicting uniqueness. So $L$ is uniquely represented by a multiplicity-free sum of Weierstrass points; since Weierstrass points lift to characteristic $0$, it follows that $L$ lifts as well.

Now, suppose $d=\rh^0(L)-1>0$. Since $\pi \colon C\to \bP^1$ is flat and $0$-dimensional, $\pi_*L$ is a rank $2$ vector bundle on $\bP^1$, and has the same cohomology as $L$. Since $\chi(L) = 0$, we also have $\chi(\pi_* L) = 0$, so it follows that
\[
\pi_*L\cong \cO_{\bP^1}(d) \oplus \cO_{\bP^1}(-d-2).
\]
We then find immediately that $\eta^{-d}\otimes L$ is an $\iota$-invariant line bundle with $\rh^0(\eta^{-d}\otimes L)=1$, so admits a lift $\tilde{L}'$.  We then find that $\tilde{L} := \eta^d \otimes \tilde{L}'$ is a lift of $L$ with at least $d+1$ global sections, so precisely $d+1$ global sections by semicontinuity. 
\end{proof}

\begin{lemma} \label{lem:theta-char}
Let $\cE$ be a vector bundle on a hyperelliptic curve $C$ of genus $g$ with $\det \cE \cong \cO_C$. Then for any theta-characteristic $\theta$ {\rm (}i.e., $\theta^2 \cong \omega_C${\rm )}, we have
\[
\rh^0(C; \cE \otimes \iota^* \cE \otimes \theta)\ge \rank(\cE) \cdot \rh^0(C; \theta).
\]
\end{lemma}

\begin{proof}
First assume that the characteristic is different from $2$. Note that $\cE \otimes \iota^*\cE \otimes \theta$ has the structure of an $\iota$-equivariant sheaf (in two ways, depending on a choice of $\iota$-equivariant structure on $\theta$). The eigenspaces on the fibers over Weierstrass points of the natural equivariant structure on $\cE \otimes \iota^* \cE$ are the symmetric and anti-symmetric forms, and this will remain true after twisting by $\theta$, except that the eigenspaces will swap at $g+1-\rh^0(C; \theta)$ Weierstrass points (for one of the two equivariant structures on $\theta$).

Let $\pi \colon C \to \bP^1$ be the double covering map and let $\pi^* \cV$ be the largest equivariant subsheaf of $\cE \otimes\iota^* \cE \otimes \theta$ that descends to $\bP^1$; $\pi^* \cV$ is the kernel of the map to the $-1$ eigenspaces of the fibers at the Weierstrass points. Since $\chi(\cE \otimes \iota^*\cE \otimes \theta)=0$, we have
\begin{align*}
\chi(\pi^*V)
&= -(g+1-2 \rh^0(C; \theta)) \frac{r(r+1)}{2} -(g+1+2 \rh^0(C; \theta)) \frac{r(r-1)}{2}\\
&= -(g+1)r^2+2 \rh^0(C;\theta)r,
\end{align*}
and thus
\[
\deg(\cV)=\deg(\pi^*\cV)/2 = -r^2+\rh^0(C; \theta) r.
\]
But then Riemann--Roch gives $\chi(\cV)=\rh^0(C; \theta)r$, so that $\cV$, thus $\pi^*\cV$, has at least that many global sections.

Now we handle the case that $\bk$ has characteristic $2$. We may as well assume that $\bk$ is algebraically closed. By Lemma~\ref{lem:lift-char2}, we can lift $(C,\theta)$ to a pair $(\tilde{C}, \tilde{\theta})$ in characteristic $0$ satisfying $\rh^0(\theta) = \rh^0(\tilde{\theta})$. Let $\tilde{\cE}$ be any lift of $\cE$ to $\tilde{C}$. Then 
\[
\rh^0(C; \cE \otimes \iota^* \cE \otimes \theta) \ge \rh^0(\tilde{C}; \tilde{\cE} \otimes \iota^* \tilde{\cE} \otimes \tilde{\theta}) \ge \rank(\tilde{\cE}) \cdot \rh^0(\tilde{C}; \tilde{\theta}) = \rank(\cE) \cdot \rh^0(C; \theta). \qedhere
\]
\end{proof}

\begin{corollary} \label{cor:restriction}
If $\cE$ is a stable rank $3$ vector bundle on $C$ with trivial determinant, then $\pgl(\cE) \otimes \cO(P)$ has a nonzero global section if and only if $\cE \cong \iota^* \cE^*$. 
\end{corollary}

\begin{proof}
Note that $\chi(\pgl(\cE) \otimes \cO(P)) = 0$, so by Serre duality and Riemann--Roch, $\pgl(\cE) \otimes \cO(P)$ has a nonzero global section if and only if the same is true for $\fsl(\cE) \otimes \cO(P)$, and we will work with the latter condition. Since $\cE$ is stable, $\rh^0(C; \fgl(\cE)) = 1$, and $\rH^0(C; \fgl(\cE))$ is spanned by the identity map $\cE \to \cE$. In particular, this has nonzero image under the trace map $\fgl(\cE) \to \cO_C$, so $\rH^0(C; \fsl(\cE)) = 0$.

Let $\sE$ be a universal sheaf on $C \times \SU^{\rm s}_3(C)$ with projection maps $p,q$. Apply $q_*$ to
\[
0 \to \fsl(\sE) \to \fsl(\sE) \otimes p^*\cO(P) \to \fsl(\sE) \otimes p^*\cO(P)|_P \to 0
\]
to get the $4$-term exact sequence ($q_* \fsl(\sE) = 0$ by what we just argued)
\begin{align*}
0 \to q_*(\fsl(\sE) \otimes p^*\cO(P)) \to q_*(\fsl(\sE) \otimes p^* \cO(P)|_P) \xrightarrow{\alpha} \rR^1 q_*(\fsl(\sE)) \qquad \qquad \\
 \to \rR^1 q_*(\fsl(\sE) \otimes p^*\cO(P)) \to 0.
\end{align*}
The locus of interest, that is, stable bundles $\cE$ where $\rH^0(C; \fsl(\cE) \otimes \cO(P)) \ne 0$, is set-theoretically defined by $\det(\alpha)$. Let $\cF$ and $\cG$ denote the source and target of $\alpha$, respectively. So $\det(\alpha)$ is a section of $\det(\cG) \otimes \det(\cF^*)$. Since $\cG = \cT_{\SU_3^{\rm s}(C)}$, its determinant is $\cO_{\SU^{\rm s}_3(C)}(6)$. 

We claim that $\det(\cF) = \cO_{\SU^{\rm s}_3(C)}$. Define $s \colon \SU^{\rm s}_3(C)\to C \times \SU^{\rm s}_3(C)$ by $s(\cE) = (P, \cE)$. For any sheaf $\cM$ on $C \times \SU^{\rm s}_3(C)$, we have $q_*(\cM|_P) = s^* \cM$. Applying this to $\cM = \fsl(\sE) \otimes p^*\cO(P)$, we get
\[
\cF = q_*(\fsl(\sE)\otimes p^*\cO(P)|_{P}) = s^*(\fsl(\sE))\otimes s^*p^*\cO(P).
\]
Note that $s^* p^* \cO(P)$ is trivial, since $p^*\cO(P)\cong \cO(\im s)$ is the line bundle associated to a section of a product. Also, $s^*(\fsl(\sE))\cong \fsl(s^* \sE)$, and $\fsl$ of any vector bundle has trivial determinant. So $\det(\cF) = \cO_{\SU^{\rm s}_3(C)}$.

So $\alpha$ is a section of $\cO_{\SU^{\rm s}_3(C)}(6)$. We have an exact sequence
\[
0 \to \rH^0(C; \fsl(\cE) \otimes \cO(P)) \to \rH^0(C; \fgl(\cE) \otimes \cO(P)) \to \rH^0(C; \cO(P)).
\]
The third term has dimension $1$, and by Lemma~\ref{lem:theta-char} with $\eta = \cO(P)$, the second term has dimension $\ge 3$ when $\cE \cong \iota^* \cE^*$. So $\rh^0(C; \fsl(\cE) \otimes \cO(P)) \ge 2$ when $\cE \cong \iota^* \cE^*$. The ramification locus is branched along a sextic in $\bP^8$, so it is a zero section of $\cO_{\SU^{\rm s}_3(C)}(3)$. So $\det(\alpha)$ contains the ramification locus with either multiplicity $1$ or $2$. But on the ramification locus, the rank of $\alpha$ drops by at least $2$, so the multiplicity is at least $2$. So we conclude that $\rH^0(C; \fsl(\cE) \otimes \cO(P)) \ne 0$ if and only if $\cE \cong \iota^* \cE^*$.
\end{proof}

\begin{proof}[{Proof of Proposition~\ref{prop:inverse1}}]
Let $\cE$ be a rank $3$ stable bundle with trivial determinant. The form $\gamma_\cE$ may be defined as the composition of either of two maps in the following commutative square:
\[
\xymatrix{ \bigwedge^3 \rH^0(C; \pgl(\cE) \otimes \omega_C) \ar[r] \ar[d] & \rH^0(C; \omega_C^{\otimes 3}) \ar[d] \\
\bigwedge^3 \rH^0(P; \pgl(\cE) \otimes \omega_C) \ar[r] & \rH^0(P; \omega_C^{\otimes 3}) }.
\]
In particular, $\gamma_\cE$ has nonzero hyperdiscriminant if and only if the restriction 
\[
\rH^0(C; \pgl(\cE) \otimes \omega_C) \to \rH^0(P; \pgl(\cE) \otimes \omega_C)
\]
is an isomorphism. By considering the short exact sequence
\[
0 \to \pgl(\cE) \otimes \cO(P) \to \pgl(\cE) \otimes \omega_C \to (\pgl(\cE) \otimes \omega_C)|_P \to 0,
\]
we see that $\rH^0(C; \pgl(\cE) \otimes \omega_C) \to \rH^0(P; \pgl(\cE) \otimes \omega_C)$ is an isomorphism if and only if $\rH^0(C; \pgl(\cE) \otimes \cO(P)) = 0$ (since $\chi(\pgl(\cE) \otimes \cO(P))=0$). By Corollary~\ref{cor:restriction}, this is equivalent to $\cE \not\cong \iota^* \cE^*$. 

Putting this together, $\gamma_\cE$ has nonzero hyperdiscriminant if and only if $\cE \not\cong \iota^* \cE^*$. So by Proposition~\ref{prop:involution}, $D_\gamma$ and the branch locus of $\theta$ are the same as sets. Both of them are degree $6$ hypersurfaces, and we will show in Corollary~\ref{cor:irred} that $D_\gamma$ is irreducible. In particular, they must be equal as schemes. 
\end{proof}

\subsection{Proof of Proposition~\ref{prop:gamma-stable}}

\begin{lemma} \label{lem:U2-normal}
The map $\phi \colon \rU_2(C)\to \SU_3(C)$ given by $E\mapsto E\oplus \det(E)^{-1}$ is the normalization of its image, which is the singular locus of $\SU_3(C)$.
\end{lemma}

\begin{proof}
As a morphism to its image, $\phi$ is quasi-finite and projective, and hence finite. Since $\phi$ is also birational, it is the normalization map.

The stable locus of $\SU_3(C)$ is smooth, and the image of $\phi$ is precisely the complement of the stable locus. Showing that the singular locus is all of the image reduces to showing that the generic point in the image of $\phi$ has larger than expected tangent space; this follows by semicontinuity from characteristic $0$.
\end{proof}

\begin{lemma} \label{lem:alb}
The determinant map $\det \colon \rU_2(C) \to \rJ(C)$ induces an isomorphism of schemes $\Alb^1(\rU_2(C)) \to \rJ(C)$.
\end{lemma}

\begin{proof}
Since $\det$ is a morphism to an abelian variety, it factors through the Albanese torsor, and it remains only to show that the induced morphism $\Alb^1(\rU_2(C))\to \rJ(C)$ is an isomorphism. The preimage of $0 \in \rJ(C)$ is a torsor over a subgroup scheme of $\Alb^0(\rU_2(C))$, so it remains to show that this torsor is a reduced point. The preimage of $0\in \rJ(C)$ in $\Alb^1(\rU_2(C))$ agrees with the image in $\Alb^1(\rU_2(C))$ of $\SU_2(C)$. Since $\SU_2(C)$ is a smooth proper variety, the map $\SU_2(C)\to \Alb^1(\rU_2(C))$ factors through $\Alb^1(\SU_2(C))$, and since $\SU_2(C)\cong \bP^3$, $\Alb^1(\SU_2(C))$ is a single point.  
\end{proof}

\begin{lemma} \label{lem:proj-frame}
If $X$ is a subvariety of $\bP^n$ such that none of its irreducible components is contained in a hyperplane, then any element of $\PGL_{n+1}$ that fixes $X$ pointwise is trivial.
\end{lemma}

\begin{proof}
Pick $n+1$ points on $X$ which linearly span $\bP^n$. For each subset of $n$ points, take the hyperplane spanned by them. By our assumptions, $X$ is not contained in the union of these $n+1$ hyperplanes. So take a point in $X$ outside of the union of these hyperplanes. This collection of $n+2$ points gives a projective coordinate frame. Indeed, if we lift the vectors to $\bk^{n+1}$, then there is a unique linear dependency up to scalar among these points and none of the coefficients can be $0$ otherwise it would contradict the construction of the last point. So any element of $\PGL_{n+1}$ that fixes $X$ pointwise also fixes these $n+2$ points, and hence must be the identity.
\end{proof}

\begin{lemma} \label{lem:autSU3}
The automorphism group of $\SU_3(C)$ is finite.
\end{lemma}

\begin{proof}
We have homomorphisms
\[
\Aut(\SU_3(C))\to \Aut({\rm Sing}(\SU_3(C))) \to \Aut(\rU_2(C))\to \Aut(\Alb^1(\rU_2(C)))\cong \Aut(\rJ(C)),
\]
where the second map comes from Lemma~\ref{lem:U2-normal} and the isomorphism comes from Lemma~\ref{lem:alb}. 

By Proposition~\ref{prop:picard-group}, $\Pic(\SU_3(C))$ is infinite cyclic and only one of its generators has sections (for example, by Corollary~\ref{cor:verlinde}), so any automorphism of $\SU_3(C)$ has to preserve $\cO(1)$, and in particular, the map $\theta$. In particular, $\Aut(\SU_3(C))$ is affine (being an extension with finite kernel of a linear group scheme). 

So its image in $\Aut(\rJ(C))$ is finite, and it suffices to show that the kernel $K$ of this composition is finite. A point $E$ in $\SU_3(C)$ that splits as a sum of three line bundles $L_1 \oplus L_2 \oplus L_3$ is determined by the subscheme of $\rJ(C)$ which is the image of $\phi^{-1}(E)$ ($\phi$ being the map in Lemma~\ref{lem:U2-normal}) in $\Alb^1(\rU_2(C))$ (geometrically, this consists of the points $\{L_1 \otimes L_2, L_2 \otimes L_3, L_1 \otimes L_3\}$). It follows therefore that $E$ is fixed by $K$, and hence $K$ pointwise fixes the locus $X$ of points in $\SU_3(C)$ representing vector bundles that are sums of three line bundles. 

Note that $X$ is irreducible since it is the image of $\rJ(C) \times \rJ(C) \to \SU_3(C)$ under the map $(L_1, L_2) \mapsto L_1 \oplus L_2 \oplus (L_1 \otimes L_2)^{-1}$. Furthermore, $\theta(X)$ is a reduced Heisenberg-invariant subscheme and so is not contained in any hyperplane of $\bP^8$: indeed, the given representation of the Heisenberg group is irreducible (even for non-ordinary curves of characteristic 3) by \cite[Appendix]{sekiguchi}. So by Lemma~\ref{lem:proj-frame}, any element of $\PGL_9$ which fixes $\theta(X)$ is trivial, and hence $K$ is finite.
\end{proof}

\begin{proof}[Proof of Proposition~\ref{prop:gamma-stable}]
First, we claim that the stabilizer in $\SL_9(\bk)$ is finite. For this, it suffices to show that the projective stabilizer of the line $[\gamma_{(C,P)}]$ in $\PGL_9(\bk)$ is finite. Note that any element in the projective stabilizer gives a linear change of coordinates of $\bP^8$ which preserves the branch locus of $\theta$, and hence gives an automorphism of $\SU_3(C)$, so we can use Lemma~\ref{lem:autSU3} to finish the proof of our claim.

If $\rJ(C)[3] \cong \bmu_3^2\times (\bZ/3)^2$, then $\gamma_{(C,P)}$ is Heisenberg-invariant, so can be conjugated into the subspace $\fh$ spanned by the vectors in \cite[(2.4)]{w39-paper}, and it follows in particular that there exist such elements with finite stabilizer.  Hence,  by dimension considerations, if $U\subset \fh$ is the open subset on which the stabilizer is finite, then $\SL_9 \cdot U$ is Zariski dense in $\bigwedge^3(V_9)$. It follows in particular that over any field, the generic point of $\fh$ is stable, and since the non-stable locus is everywhere of codimension at most $1$ \cite[Proposition 2.10]{w39-paper}, it follows that the non-stable locus in $\fh$ over $\bZ$ is the Zariski closure of the characteristic $0$ non-stable locus.  The latter is a union of hyperplanes \cite[Proposition 2.8]{w39-paper}, and thus the same holds over any field.  Any element of such a hyperplane can clearly be lifted to characteristic $0$ in such a way as to remain a non-stable element of the $\fh$.  In particular, the lift has infinite stabilizer, so the same holds for the original element. 

It follows that $\gamma_{(C,P)}$ is stable for all but the codimension $\ge 2$ locus of non-ordinary curves of characteristic $3$. Again, \cite[Proposition 2.10]{w39-paper} tells us that the non-stable locus is everywhere of codimension at most $1$, and must therefore be empty in the moduli stack of pairs $(C,P)$ with $C$ smooth. 
\end{proof}

\section{Projective duality and Coble hypersurfaces}

Let $V_9$ be a $9$-dimensional vector space. Fix a stable vector $\gamma \in \bigwedge^3 V_9$ with respect to the natural action of $\SL(V_9)$. Let $\bP(V_9^*)$ denote the space of lines in $V_9^*$. Apply the comultiplication map $\bigwedge^3 V_9 \to \bigwedge^2 V_9 \otimes V_9$ to $\gamma$, use the natural surjection $V_9 \otimes \cO_{\bP(V_9^*)} \to \cO_{\bP(V_9^*)}(1)$, and interpret $\bigwedge^2 V_9$ as the space of skew-symmetric matrices $V_9^* \to V_9$. This allows us to interpret $\gamma$ as a family of skew-symmetric matrices
\[
\Phi_\gamma \colon V_9^* \to V_9 \otimes \cO_{\bP(V_9^*)}(1)
\]
over $\bP(V_9^*)$. Let $X_\gamma \subset Y_\gamma \subset \bP(V_9^*)$ be the loci
\begin{align*}
X_\gamma &= \{x \in \bP(V_9^*) \mid \rank(\Phi_\gamma|_x) \le 4\},\\
Y_\gamma &= \{x \in \bP(V_9^*) \mid \rank(\Phi_\gamma|_x) \le 6\}.
\end{align*}
Improving upon \cite{GS,GSW}, the following theorem is proven in \cite[\S 3]{w39-paper}.

\begin{theorem} \label{thm:w39-main}
\begin{enumerate}[\rm \indent (a)]
\item $Y_\gamma$ is a cubic hypersurface whose singular locus is $X_\gamma$.
\item $X_\gamma$ is smooth of dimension $2$, and the locus where $\rank \Phi_\gamma \le 2$ is empty.
\item If $X_\gamma$ has a rational point, then it is isomorphic to the Jacobian of a smooth genus $2$ curve $C_\gamma$. Furthermore, the line bundle giving the embedding $X_\gamma \subset \bP(V_9^*)$ is a $(3,3)$ polarization.
\end{enumerate}
\end{theorem}

In Definition~\ref{def:Dgamma}, we constructed a sextic hypersurface $D_\gamma \subset \bP(V_9)$.

\begin{proposition} \label{prop:proj-dual}
If $\gamma$ is stable, then $D_\gamma$ and $Y_\gamma$ are projectively dual hypersurfaces.
\end{proposition}

\begin{proof}
We may assume, without loss of generality, that $\bk$ is algebraically closed. We introduce an auxiliary variety. Let $F$ be the variety of flags of the form $V_1 \subset V_3 \subset V_8 \subset V_9^*$ (the subscripts denote dimensions). %Then $\dim F = 25$. 
Let $\cV_1 \subset \cV_3 \subset \cV_8 \subset \cV_9 = V_9^* \otimes \cO_F$ denote the tautological flag of subbundles on $F$. Then $\bigwedge^3 \cV_9$ has a homogeneous subbundle $\xi$ which has a filtration by (homogeneous) subbundles whose associated graded is
\[
\bigwedge^3 \cV_3 \oplus (\cV_1 \otimes \cV_3/\cV_1 \otimes \cV_9/\cV_3) \oplus (\bigwedge^2(\cV_3/\cV_1) \otimes \cV_8 / \cV_3).
\]
Let $Z_\gamma$ be the subvariety of $F$ where the image of $\gamma$ under $\bigwedge^3 \cV^*_9 \to \xi^*$ is $0$. Let $\pi_1 \colon F \to \bP(V_9^*)$ and $\pi_2 \colon F \to \bP(V_9)$ be the maps sending $V_1 \subset V_3 \subset V_8$ to $V_1$ and $V_8$, respectively.

Pick a point $x \in Z_\gamma$ and choose a basis $e_1, \dots, e_9$ for $V_9$ so that $x$ is the standard coordinate flag ($V_i$ is spanned by $e_1^*, \dots, e_i^*$). Write $[ijk]$ in place of $e_i \wedge e_j \wedge e_k$. The conditions imposed by the composition $\bigwedge^3 \cV_9^* \to \xi^*$ being $0$ are that the coefficients of the following monomials are $0$:
\begin{align*}
[123], \qquad [12i], [13i] \quad (4\le i\le 9) \qquad [23j] \quad (4 \le j \le 8)
\end{align*}
In particular, the $9 \times 9$ skew-symmetric matrix $\Phi_\gamma(\pi_1(x))$ only has nonzero entries in the bottom right $6 \times 6$ corner, so has rank $\le 6$, and so $\pi_1(x) \in Y_\gamma$. Conversely, given a point $y \in Y_\gamma \setminus X_\gamma$ (so $\rank \Phi_\gamma(y) = 6$), we may choose a basis so that this point is given by $e_i(y)=0$ for $i>1$ and the elements in its kernel satisfy $e_i=0$ for $i>3$. We may also arrange for the coefficients of $[23j]$ to vanish for $j=4,\dots,8$. Then the standard coordinate flag in $F$ maps to $y$. So $\pi_1(Z_\gamma) = Y_\gamma$. 

We claim that the tangent space to $y$ in $Y_\gamma$ is $V_8/V_1$. In our coordinates, $y = [1:0:0:0:0: 0:0:0:0]$, so we work in the affine subspace given by $e_1 = 1$. The equation of the tangent space is $\sum_i \frac{\partial f}{\partial e_i}(0) e_i$ where $f$ is the Pfaffian of $\Phi_\gamma(y)$ after specializing $e_1=1$ (and deleting the first row and column). Let $M$ be this $8 \times 8$ skew-symmetric matrix (so $f = \Pf(M)$). Using the Laplace expansion for Pfaffians, we can compute $f$ by taking a signed sum of $M_{1,i} P_i$ where $2 \le i \le 8$ and $P_i$ is the Pfaffian of the result of deleting rows and columns $1$ and $i$ from $M$. Note that $P_i(0) = 0$ unless $i=2$ since the entries of $\Phi_2(y)$ are concentrated in the bottom right $6 \times 6$ matrix. Since $e_1$ does not appear in rows 1 and 2 (before specializing $e_1=1$), this implies that the equation of the tangent space is $\sum_i \frac{\partial M_{1,2} P_2}{\partial e_i}(0) e_i$. From the constraints on $\gamma$ above, $M_{1,2}$ is the coefficient of $[239]$ times $e_9$. In particular, the equation of the tangent space is a (nonzero, since $y$ is smooth) multiple of $e_9=0$, so the claim is proven.

In particular, $\pi_1$ is injective over the smooth locus (we have shown that $V_3$ and $V_8$ in any flag in $\pi_1^{-1}(y)$ are both determined by $y$). Since $Y_\gamma$ is normal (its singular locus $X_\gamma$ has codimension $5$), this implies that $\pi_1$ is birational.

The image of $\gamma$ under the projection $\bigwedge^3 \cV_9^* \to \bigwedge^3(\cV_8)^*_x$ has the property that $\gamma(v_1,v_2,v_3) = 0$ whenever two of the vectors belong to $(\cV_3)_x$. But $\cV_8^*$ is the pullback along $\pi_2$ of the bundle $\cQ$ in Definition~\ref{def:Dgamma}, so this is the same trilinear form used there. In particular, the hyperdiscriminant of this form vanishes by Proposition~\ref{prop:w38-kernel}, so $\pi_2(x) \in D_\gamma$.

Conversely, given a point $y \in D_\gamma$ whose associated $3$-form has vanishing hyperdiscriminant, there is a $3$-dimensional space $V_3$ with the property that $\gamma(v_1,v_2,v_3)=0$ whenever two of the vectors belong to $V_3$ (Proposition~\ref{prop:w38-kernel}). As above, we can pick a basis so that the standard coordinate flag lies in $Z_\gamma$ and maps to $y$ under $\pi_2$. So $\pi_2(Z_\gamma) = D_\gamma$.

In particular, we conclude that $D_\gamma$ is the closure of the set of tangent hyperplanes to smooth points of $Y_\gamma$, so the two hypersurfaces are projectively dual to one another.
\end{proof}

\begin{corollary} \label{cor:irred}
$Y_\gamma$ and $D_\gamma$ are irreducible hypersurfaces.
\end{corollary}

\begin{proof}
If $Y_\gamma$ were reducible, then it would be singular in codimension 1, but we know that its singular locus is $X_\gamma$, which has codimension 5 in $Y_\gamma$. The projective dual of an irreducible variety is irreducible \cite[\S 1.1, Proposition 1.3]{GKZ}, so $D_\gamma$ is also irreducible.
\end{proof}

\begin{theorem} \label{thm:coble-duality}
Let $C$ be a smooth genus $2$ curve. In characteristic different from $2$, the branch locus of $\theta \colon \SU_3(C) \to |3\Theta|$ is a degree $6$ hypersurface whose projective dual is a $\rJ(C)[3]$-invariant cubic hypersurface whose singular locus is a torsor for $\rJ(C)$.
\end{theorem}

\begin{proof}
We may as well assume that $\bk$ is algebraically closed. From Corollary~\ref{cor:irred}, the branch locus is irreducible. Since the characteristic is different from $2$, irreducibility of $\SU_3(C)$ combined with this fact implies that the branch locus is a reduced hypersurface. 

Pick any Weierstrass point $P \in C$. The construction in \S\ref{ss:construct-trivector} gives a vector $\gamma_{(C,P)} \in \bigwedge^3 V_9$ with $V_9 = \rH^0(\SU_3(C); \Theta)^*$ so that $D_{\gamma_{(C,P)}}$ is the branch locus of $\theta$ (Proposition~\ref{prop:inverse1}). Furthermore, $\gamma_{(C,P)}$ is stable by Proposition~\ref{prop:gamma-stable}. By Proposition~\ref{prop:proj-dual}, the projective dual of $D_{\gamma_{(C,P)}}$ is a cubic hypersurface, and from Theorem~\ref{thm:w39-main}, the singular locus of the cubic hypersurface is isomorphic to $\rJ(C)$.
\end{proof}

\begin{remark}
In characteristics different from $3$, it is easy to show using the techniques in \cite{beauville} that there is a unique $\rJ(C)[3]$-invariant cubic hypersurface whose singular locus is $\rJ(C)$. In particular, Theorem~\ref{thm:coble-duality} generalizes the main result of \cite{minh, ortega} to characteristics different from $2$ and $3$.
\end{remark}

\begin{remark}
We have seen how to construct the Coble cubic and its projective dual directly from $\gamma$. Using the equation \eqref{eqn:double}, we can also build, directly from $\gamma$, a double cover of $\bP(V_9)$ which is branched along $D_\gamma$. However, we also know that this double cover is $\SU_3(C)$, so we have an invariant-theoretic construction for it.
\end{remark}

Finally, we make some comments about the case when $\bk$ is a field of characteristic $2$.

\begin{lemma} \label{lem:exterior-quad}
Let $V$ be a vector space over a field of characteristic $2$, and let $k>1$ be an integer. Then there is a unique quadratic map $Q \colon \bigwedge^k V\to \bigwedge^{2k} V$ which vanishes on pure tensors and satisfies $Q(v+w)=Q(v)+Q(w)+v\wedge w$.
\end{lemma}

\begin{proof}
Uniqueness is easy: a quadratic map is determined by the associated bilinear map and its values on a basis, and the standard basis consists entirely of pure tensors.  

For existence, define $Q$ by $Q(v_1 + \cdots + v_r) = \sum_{1 \le i < j \le r} v_i \wedge v_j$ whenever the $v_i$ are pure tensors in the standard basis. It remains only to show that $Q$ vanishes on all pure tensors, not just those formed from coordinate vectors. This follows by induction on the number of coordinate vectors that appear in a pure tensor: expand one of the non-coordinate vectors in the wedge into coordinate vectors and observe that each term in the sum is annihilated by $Q$ and any two terms wedge to $0$.
\end{proof}

\begin{remark} \label{rmk:char2}
If $\bk$ is a field of characteristic $2$, the branch locus of the map $\theta \colon \SU_3(C) \to \bP(V_9)$ is the square of a cubic equation due to Lemma~\ref{lem:squaremod4}. We will identify this cubic equation now.

Let $\fe_8$ be the split Lie algebra of type $\rE_8$ over $\bk$. It has a $\bZ/3$-graded decomposition $\fsl(V_9) \oplus \bigwedge^3 V_9 \oplus \bigwedge^6 V_9$ (see \cite[\S 2.2]{w39-paper}). Also, $\fe_8$ has a squaring map $x \mapsto x^{[2]}$ which induces a squaring map $\bigwedge^3 V_9 \to \bigwedge^6 V_9$; this is the map $Q \colon \bigwedge^3 V_9 \to \bigwedge^6 V_9$ defined in Lemma~\ref{lem:exterior-quad}. Given a quotient $V_9 \to V_8$, also use $Q$ to denote the squaring map $\bigwedge^3 V_8 \to \bigwedge^6 V_8$. Then we have a commutative square
\[
\xymatrix{ \bigwedge^3 V_9 \ar[r]^-Q \ar[d]_-\pi & \bigwedge^6 V_9 \ar[d]^-\pi \\
\bigwedge^3 V_8 \ar[r]^-Q & \bigwedge^6 V_8 }.
\]
We claim that $f_{16}(\pi(\gamma)) = \Pf(Q(\pi(\gamma)))^2$ (recall that $f_{16}$ is the hyperdiscriminant of $\bigwedge^3 V_8$). Both $f_{16}$ and $(\Pf \circ Q)^2$ are degree $16$ invariants; note that $\bigwedge^3 V_8$ has a dense orbit with respect to $\GL(V_8)$, so by uniqueness, there is a constant $c$ such that $(\Pf \circ Q)^2 = cf_{16}$. Since both are defined over $\bF_2$, we have $c \in \{0,1\}$, so we just need to show that $\Pf \circ Q$ is not identically $0$. First note that $Q([123] + [456]) = [123456]$, so $Q \ne 0$. It follows that if $\pi(\gamma)$ is an element in the dense orbit given by \eqref{eqn:trace-form}, then $Q(\pi(\gamma)) \ne 0$. The stabilizer of this element in $\fsl(V_8)$ is $\fsl_3$, which acts irreducibly on $V_8$ (it is the adjoint representation), so we conclude that $Q(\pi(\gamma))$ has full rank, and so $\Pf(Q(\pi(\gamma))) \ne 0$. This shows that $c=1$ and proves the claim.

A point in $\bP(V_9)$ represents a quotient $V_9 \to V_8$, and it belongs to $D_\gamma$ if and only if $f_{16}(\pi(\gamma)) = 0$. From our claim above, this is equivalent to $\Pf(Q(\pi(\gamma))) \ne 0$.

So the cubic equation whose square is the branch locus of $\theta$ can be obtained by applying the Pfaffian construction (from the beginning of this section) to $\gamma^{[2]} \in \bigwedge^3 V_9^* \cong \bigwedge^6 V_9$.
\end{remark}

\appendix

\section{Comparison}

Keep the notation from \S\ref{sec:genus2-gen}. In \S\ref{ss:construct-trivector}, we built a trivector $\gamma \in \bigwedge^3 V_9$. Via comultiplication, we turn this into a map $V_9^* \to \bigwedge^2 V_9$. Let $W_\gamma \subset \bigwedge^2 V_9$ be its image. In \cite[\S 4]{w39-paper}, we give an alternative description for $W_\gamma$. Our goal in this appendix is to prove Theorem~\ref{thm:wgamma-eqn} so that we can show that this description agrees with the one here.

Let $V_9 = \rH^0(\SU_3(C); \cO(1))^*$ so that $\rJ^1(C) \subset \bP(V_9^*)$ is embedded by a $(3,3)$-polarization, denoted $\cO(1)$. We identify $\rJ^1(C) \cong \rJ(C)$ via $\cL \mapsto \cL(-P)$. Define a codimension $1$ subvariety of $\rJ(C) \times \rJ(C)$ by
\[
X = X_{C,P} = \{(\cL_1, \cL_2) \mid \hom_C(\cL_1, \cL_2(P)) \ne 0\}.
\]
The line bundle $\cO(1,1) \otimes \cO(-X)$ has divisor class $3\pi_1^* \Theta + 3\pi_2^* \Theta - \Theta_{\rm diag}$. This is the pullback of a principal polarization on $\rJ(C) \times \rJ(C)$ via the endomorphism 
\begin{align*}
\rJ(C) \times \rJ(C) &\to \rJ(C) \times \rJ(C)\\
(a,b) &\mapsto (2a+b, a+2b).
\end{align*}
The kernel of this map is the diagonal copy of $\rJ(C)[3]$ which has degree $81$. In particular, $\cO(1,1) \otimes \cO(-X)$ has a single cohomology group of dimension $9 = \sqrt{81}$. We will see that it has nonzero global sections by exhibiting bilinear equations that vanish on $X$.

\begin{theorem} \label{thm:wgamma-eqn}
\begin{enumerate}[\rm (a)]
\item $\dim W_\gamma = 9$.
\item $W_\gamma \subset \rH^0(\rJ(C) \times \rJ(C); \cO(1,1))$ vanishes on $X$. In particular, 
\[
W_\gamma = \rH^0(\rJ(C) \times \rJ(C); \cO(1,1) \otimes \cO(-X)).
\]
\end{enumerate}
\end{theorem}

If $\cL$ is a line bundle of degree $0$ on $C$, then the locus 
\begin{align} \label{eqn:theta}
\Theta_{\cL;\SU_3(C)} = \{\cE\in \SU_3(C) \mid \hom_C(\cL,\cE(P))\ne 0\}
\end{align}
is a double cover of a hyperplane in $\bP^8$. In particular, it is smooth away from the ramification locus, so that $\cL$ determines a codimension $1$ subspace of the tangent space at $\cE$, and thus an element $\lambda_\cL$ of the cotangent space (modulo scalars). 

Generically, $\Theta_\cE$ is a reduced divisor of $\rJ^1(C)$, and so $\dim\hom(\cL,\cE(P))=1$. By Serre duality and Riemann--Roch, $\dim\hom(\cE,\cL(P))=1$, and thus for generic $\cE$, we have a unique (modulo scalars) composition
\[
\tilde{\lambda}_\cL \colon \cE\to \cL(P)\to \cE(2P)\cong \cE\otimes\omega_C.
\]

\begin{lemma}
Assume that $\dim \hom(\cE, \cL(P)) = 1$. The image of $\tilde{\lambda}_\cL$ under 
\[
\rH^0(C; \fgl(\cE) \otimes \omega_C) \to \rH^0(C; \pgl(\cE) \otimes \omega_C)
\]
is $\lambda_\cL$ $($up to nonzero scalar$)$.
\end{lemma}

\begin{subeqns}
\begin{proof}
A standard calculation with cocycles over the dual numbers (see \cite[Proposition 2.6]{mukai}) shows that the tangent space of $\Theta_{\cL;\SU_3(C)}$ at $\cE$ is the kernel of the cup product with any nonzero element in $\hom(\cE, \cL(P))$:
\begin{align} \label{eqn:cup-prod}
\rH^1(C; \fsl(\cE)) \subset \ext^1(\cE,\cE)\to \ext^1(\cE,\cL(P)).
\end{align}
Twist the map $\cL(P) \to \cE(2P)$ by $\omega_C^{-1}$; this gives a nonzero map $\hom(\cE, \cE) \to \hom(\cL(-P), \cE)$, which is an isomorphism since both are $1$-dimensional ($\dim \hom(\cE,\cE) = 1$ since $\cE$ is stable). By Serre duality, the dual map $\ext^1(\cE, \cL(P)) \to \ext^1(\cE, \cE \otimes \omega_C)$ is also an isomorphism. So being in the kernel of \eqref{eqn:cup-prod} is equivalent to being in the kernel of the cup product map with $\tilde{\lambda}_\cL$:
\[
\rH^1(C; \fsl(\cE)) \subset \ext^1(\cE,\cE)\to \ext^1(\cE,\cE \otimes \omega_C).
\]
In other words, $\tilde{\lambda}_\cL$ induces a linear functional on $\rH^1(C; \fsl(\cE))$ whose kernel is the tangent space to $\Theta_{\cL; \SU_3(C)}$ and so agrees with $\lambda_\cL$ up to nonzero scalar multiple.
\end{proof}
\end{subeqns}

\begin{lemma} \label{lem:cL-properties}
Let $\cL_1, \cL_2$ be line bundles such that $\dim \hom(\cL_i, \cE(P)) = 1$ so that $\tilde{\lambda}_{\cL_i}(P)$ are defined.
\begin{enumerate}[\rm (1)]
\item $\tilde\lambda_{\cL_i}(P)$ is a rank $\le 1$ endomorphism of the fiber $\cE_P$.
\item If $\dim\hom(\cL_1,\cL_2(P))=1$, then $\tilde\lambda_{\cL_1}(P)$ and $\tilde\lambda_{\cL_2}(P)$ commute.
\end{enumerate}
\end{lemma}

\begin{proof}
(1) is an immediate consequence of the fact that $\tilde\lambda_{\cL_i}$ factors through a line bundle. 

(2) is vacuous if $\cL_1 \cong \cL_2$. Otherwise, it suffices to show that both compositions vanish. Consider the composition
\[
\cE\to \cL_1(P)\to \cE(2P)\to \cL_2(3P)\to \cE(4P).
\]
Since $\cL_1\not\cong \cL_2$, Riemann--Roch gives $\dim\hom(\cL_1(P), \cL_2(3P)) = 1$. By our assumption, it follows that the above composition factors through $\cL_2(2P) \to \cL_2(3P)$, which is $0$ when restricted to $P$. The other composition is $0$ by the same argument if we swap $\cL_1$ and $\cL_2$; we note that $\dim \hom(\cL_2, \cL_1(P)) = 1$ by Serre duality and Riemann--Roch.
\end{proof}

\begin{lemma} \label{lem:generic-pair}
For a generic pair $(\cL_1,\cL_2) \in \rJ(C) \times \rJ(C)$ with $\dim \hom(\cL_1,\cL_2(P))=1$, there exists a stable rank $3$ bundle $\cE$ with $\cE \not\cong \iota^* \cE^*$ such that $\dim\hom(\cL_1,\cE(P)) = \dim\hom(\cL_2,\cE(P))=1$.  
\end{lemma}

\begin{proof}
Let $X$ be the set of triples $(\cL_1,\cL_2,\cE) \in \rJ(C) \times \rJ(C) \times \SU_3(C)$ where $\hom(\cL_1, \cL_2(P)) \ne 0$, $\hom(\cL_1, \cE(P)) \ne 0$, and $\hom(\cL_2, \cE(P)) \ne 0$.  Let $X_0$ be an irreducible component of $X$ whose image in $\rJ(C)\times \rJ(C)$ dominates the divisor $\{(\cL_1, \cL_2) \mid \hom_C(\cL_1, \cL_2(P)) \ne 0\}$.

Now, let $(\cL_1,\cE)$ be a pair with $\dim\hom(\cL_1,\cE(P))=1$ and $\dim\hom(\cE,\iota^*\cE^*)=0$ (such a pair exists, since each condition is separately a nonempty open condition on the locus $\{(\cL_1, \cE) \mid \hom(\cL_1, \cE(P)) \ne 0\}$ inside $\rJ(C) \times \SU_3(C)$).  The loci $\{\cL_2 \mid \hom(\cL_1,\cL_2(P))\ne 0\}$ and $\{\cL_2 \mid \hom(\cL_2,\cE(P))\ne 0\}$ are sections of $\Theta_{\cL_1}$ and $3\Theta$, respectively, and thus have intersection multiplicity $3 \Theta^2 = 6$.  Moreover, there exists a section of $3\Theta$ meeting the given section of $\Theta_{\cL_1}$ at the point representing $\cL_1$ with multiplicity $1$ (take any hyperplane through $\cL_1$ not containing the tangent vector), so this is true for the generic such section, and can be imposed as a further nonempty open condition on $\cE$. We thus see that any such pair $(\cL_1,\cE)$ can be extended to a triple $(\cL_1,\cL_2,\cE)$ in $X_0$ (note that $\hom(\cL_1, \cL_2(P)) \ne 0$ implies that $\dim \hom(\cL_1, \cL_2(P)) = 1$).

But then the locus in $X_0$ for which 
\[
\dim\hom(\cL_1,\cL_2(P)) = \dim\hom(\cL_1,\cE(P)) = \dim\hom(\cL_2,\cE(P)) = 1
\]
and $\dim\hom(\cE,\iota^*\cE^*)=0$ is an intersection of nonempty open sets, and is thus (by irreducibility) nonempty.
\end{proof}

\begin{proof}[Proof of Theorem~\ref{thm:wgamma-eqn}]
(a) If not, then there is an $8$-dimensional subspace $V_8 \subset V_9$ so that $\gamma \in \bigwedge^3 V_8$. In this case, $D_\gamma$ contains the hyperplane $\bP(V_8)$. But this cannot happen: in odd characteristic, $D_\gamma$ is a reduced irreducible sextic by Corollary~\ref{cor:irred} and Theorem~\ref{thm:coble-duality}, while in even characteristic, it is the square of a reduced irreducible cubic by Remark~\ref{rmk:char2}.

(b) A line bundle $\cL \in \rJ(C)$ corresponds to a hyperplane in $V_9$. The preimage of this hyperplane is $\Theta_{\cL; \SU_3(C)}$ as defined in \eqref{eqn:theta} (recall that we used $P$ to identify $\rJ(C)$ with $\rJ^1(C)$). So it suffices to show that for generic choice of $(\cL_1, \cL_2)$, the linear map $\gamma(\lambda_{\cL_1}, \lambda_{\cL_2}, -)$ is identically $0$. Furthermore, since $\gamma$ is alternating, it suffices to check this on the subspace $\Theta_{\cL_1; \SU_3(C)} \cap \Theta_{\cL_2; \SU_3(C)}$. By Lemma~\ref{lem:generic-pair}, there is a stable bundle $\cE \in \SU_3(C)$ with $\cE \not\cong \iota^* \cE^*$ such that 
\[
\dim\hom(\cL_1,\cL_2(P)) = \dim\hom(\cL_1,\cE(P)) = \dim\hom(\cL_2,\cE(P))=1.
\]
By Lemma~\ref{lem:cL-properties}, $\tilde{\lambda}_{\cL_i} \colon \cE \to \cE \otimes \omega_C$ commute. So, for any element $\alpha \in \rH^0(C; \pgl(\cE)\otimes\omega_C)$, the section $\trace([\lambda_{\cL_1},\lambda_{\cL_2}] \alpha)$ vanishes at $P$. So $\gamma(\lambda_{\cL_1}, \lambda_{\cL_2}, -)$ vanishes for generic $\cE$ in $\Theta_{\cL_1;\SU_3(C)}\cap \Theta_{\cL_2;\SU_3(C)}$, so we are done.
\end{proof}


\begin{thebibliography}{GSW}

\bibitem[BM]{balaji-mehta} T. E. Venkata Balaji, V. B. Mehta, Singularities of moduli spaces of vector bundles over curves in characteristic $0$ and $p$, Special volume in honor of Melvin Hochster, {\it Michigan Math. J.} {\bf 57} (2008), 37--42.

\bibitem[Ba]{barth} W.~Barth, Quadratic equations for level-$3$ abelian surfaces, {\it Abelian varieties (Egloffstein, 1993)}, 1--18, de Gruyter, Berlin, 1995.

\bibitem[Be]{beauville} Arnaud Beauville, The Coble hypersurfaces, {\it C. R. Math. Acad. Sci. Paris} {\bf 337} (2003), no.~3, 189--194, \arxiv{math/0306097v1}.

\bibitem[Bh]{bhosle} Usha Bhosle, Pencils of quadrics and hyperelliptic curves in characteristic two, {\it J. Reine Angew. Math.} {\bf 407} (1990), 75--98.

\bibitem[Co]{coble} Arthur B. Coble, Point sets and allied Cremona groups. III, {\it Trans. Amer. Math. Soc.} {\bf 18} (1917), no.~3, 331--372. 

\bibitem[DR]{desale-ramanan} U.~V. Desale, S.~Ramanan, Classification of vector bundles of rank $2$ on hyperelliptic curves, {\it Invent. Math.} {\bf 38} (1976/77), no.~2, 161--185.

\bibitem[DN]{drezet} J.-M. Drezet, M.~S. Narasimhan, Groupe de Picard des vari\'et\'es de modules de fibr\'es semi-stables sur les courbes alg\'ebriques, {\it Invent. Math.} {\bf 97} (1989), no.~1, 53--94. 

\bibitem[GKZ]{GKZ} I. M. Gelfand, M. M. Kapranov, A. V. Zelevinsky, {\it Discriminants, resultants and multidimensional determinants}, Reprint of the 1994 edition. Modern Birkh\"auser Classics, Birkh\"auser Boston, Inc., Boston, MA, 2008.

\bibitem[GS]{GS} Laurent Gruson, Steven~V Sam, Alternating trilinear forms on a $9$-dimensional space and degenerations of $(3, 3)$-polarized Abelian surfaces, {\it Proc. London Math. Soc. (3)} {\bf 110} (2015), no.~3, 755--785, \arxiv{1301.5276v2}.

\bibitem[GSW]{GSW} Laurent Gruson, Steven~V Sam, Jerzy Weyman, Moduli of Abelian varieties, Vinberg $\theta$-groups, and free resolutions, {\it Commutative Algebra} (edited by Irena Peeva), 419--469, Springer, 2013, \arxiv{1203.2575v2}.

\bibitem[Ho]{hoffmann} Norbert Hoffmann, The Picard group of a coarse moduli space of vector bundles in positive characteristic, {\it Cent. Eur. J. Math.} {\bf 10} (2012), no.~4, 1306--1313, \arxiv{1204.4418v1}.

\bibitem[HL]{HL} Daniel Huybrechts, Manfred Lehn, {\it The Geometry of Moduli Spaces of Sheaves}, second edition, Cambridge Mathematical Library, Cambridge University Press, Cambridge, 2010. 

\bibitem[Lan]{langer} Adrian Langer, Moduli spaces of sheaves in mixed characteristic, {\it Duke Math. J.} {\bf 124} (2004), no.~3, 571--586.

\bibitem[Las]{laszlo} Yves Laszlo, Local structure of the moduli space of vector bundles over curves, {\it Comment. Math. Helvetici} {\bf 71} (1996), 373--401.

\bibitem[LR]{luna-richardson} D.~Luna, R.~W. Richardson, A generalization of the Chevalley restriction theorem, {\it Duke Math. J.} {\bf 46} (1979), no.~3, 487--496.
  
\bibitem[Mu]{mukai} Shigeru Mukai, Vector bundles and Brill-Noether theory, {\it Current topics in complex algebraic geometry (Berkeley, CA, 1992/93)}, 145--158, Math. Sci. Res. Inst. Publ. {\bf 28}, Cambridge Univ. Press, Cambridge, 1995. 

\bibitem[Mi]{minh} Nguy$\tilde{\hat{{\re}}}$n Quang Minh, Vector bundles, dualities and classical geometry on a curve of genus two, {\it Internat. J. Math.}  {\bf 18} (2007), no.~5, 535--558, \arxiv{math/0702724v1}.

\bibitem[N]{newstead} P.~E. Newstead, Stable bundles of rank $2$ and odd degree over a curve of genus $2$, {\it Topology} {\bf 7} (1968), 205--215.

\bibitem[O]{ortega} Angela Ortega, On the moduli space of rank $3$ vector bundles on a genus $2$ curve and the Coble cubic, {\it J. Algebraic Geom.} {\bf 14} (2005), no.~2, 327--356.

\bibitem[P]{popa} Mihnea Popa, Generalized theta linear series on moduli spaces of vector bundles on curves, {\it Handbook of moduli. Vol. III}, 219--255, Adv. Lect. Math. (ALM) {\bf 26}, Int. Press, Somerville, MA, 2013. 

\bibitem[RS]{w39-paper} Eric M. Rains, Steven V Sam, Invariant theory of $\bigwedge^3(9)$ and genus 2 curves, \arxiv{1702.04840v1}.

\bibitem[Ra]{raynaud} Michel Raynaud, Sections des fibr\'es vectoriels sur une courbe, {\it Bull. Math. Soc. France} {\bf 110} (1982), 103--125.

\bibitem[SK]{satokimura} M. Sato, T. Kimura, A classification of irreducible prehomogeneous vector spaces and their relative invariants, {\it Nagoya Math. J.} {\bf 65} (1977), 1--155.

\bibitem[Se]{sekiguchi} Tsutomu Sekiguchi, On projective normality of abelian varieties II, {\it J. Math. Soc. Japan} {\bf 29} (1977), no.~4, 709--727.
  
\bibitem[T]{tevelev} Evgueni Tevelev, Projectively dual varieties, \arxiv{math/0112028v1}.

\end{thebibliography}
\end{document}